\documentclass{amsart}
\usepackage{amsmath}
\usepackage{amssymb}
\usepackage{amsthm}
\usepackage{enumerate}
\usepackage[dvips]{graphicx}
\usepackage{pb-diagram}
\theoremstyle{definition}
\newtheorem{definition}{Definition}[section]
\theoremstyle{plain}
\newtheorem{lem}[definition]{Lemma}
\newtheorem{thm}[definition]{Theorem}
\newtheorem{prop}[definition]{Proposition}
\newtheorem{cor}[definition]{Corollary}
\theoremstyle{remark}

\newcommand{\cdfr}{C_{\text{df}}^r}
\newcommand{\cdfz}{C_{\text{df}}^0}
\newcommand{\csaz}{C_{\text{s.a.}}^0}
\newcommand{\rspec}{\operatorname{Spec}_r}
\newcommand{\nash}{\mathcal N}
\newcommand{\gl}{\operatorname{GL}}

\newcommand{\slin}{\operatorname{SL}}
\newcommand{\sk}{\operatorname{SK}_1}
\newcommand{\comm}{\operatorname{Comm}}
\newcommand{\supp}{\operatorname{supp}}

\makeatletter
\@namedef{subjclassname@2010}{%
  \textup{2010} Mathematics Subject Classification}
\makeatother

\begin{document}

\title[Whitehead group]{Whitehead group of the ring of smooth functions definable in an o-minimal structure}
\author[M. Fujita]{Masato Fujita}
\address{Department of Liberal Arts,
Japan Coast Guard Academy,
5-1 Wakaba-cho, Kure, Hiroshima 737-8512, Japan}
\email{fujita.masato.p34@kyoto-u.jp}

\begin{abstract}
The Whitney group $K_1(A)$ is isomorphic to $A^\times \times \sk(A)$ for some subgroup $\sk(A)$, where $A$ is a commutative ring and $A^\times$ denotes the set of units in $A$.
Consider an o-minimal expansion of a real closed field $\mathcal R=(R,0,1,+,\cdot,\ldots)$.
Let $M$ be an affine definable $C^r$ manifold, where $r$ is a nonnegative integer. 
We demonstrate its homotopy theorem and that the group $\sk(\cdfr(M))$ is isomorphic to $\sk(\cdfz(M))$, where $\cdfr(M)$ denotes the ring of definable $C^r$ functions on $M$.
\end{abstract}



\maketitle

\section{Introduction}\label{sec:intro}
O-minimal structure was initially a subject of mathematical logic, but it was realized that an o-minimal structure provides an excellent framework of geometry \cite{vdD}.
It is considered to be a generalization of semialgebraic geometry \cite{BCR}, and geometric assertions on semialgebraic sets are generalized to the o-minimal case, such as triangulation and trivialization \cite{vdD}. 

The Grothendieck ring $K_0(\csaz(M))$ of the ring of continuous semialgebraic functions on an affine semialgebraic manifold $M$ corresponds to a semialgebraic vector bundle over $M$.
The basic properties of $K_0(\csaz(M))$ is investigated in \cite[Section 12.7, Section 15.1]{BCR}. 
We can generalize the results to the o-minimal case \cite{Fujita}.
On the other hand, the Whitehead group $K_1(\cdfz(M))$ is not investigated in both the semialgebraic and o-minimal cases.

The Whitney group $K_1(A)$ is isomorphic to $A^\times \times \sk(A)$ for some subgroup $\sk(A)$, where $A$ is a commutative ring and $A^\times$ denotes the set of units in $A$.
We study the group $\sk(A)$ in the o-minimal setting.
The assertions similar to those on $K_0(\cdfr(M))$ hold true for $\sk(\cdfr(M))$.
We demonstrate its homotopy theorem and that the group $\sk(\cdfr(M))$ is isomorphic to $\sk(\cdfz(M))$, where $\cdfr(M)$ denotes the ring of definable $C^r$ functions on $M$.

The paper is organized as follows.
We first review the definitions in Section \ref{sec:def}. 
We show that the group $\sk(\cdfr(M))$ is isomorphic to $\sk(\cdfz(M))$ in Section \ref{sec:main1}.
Finally, we demonstrate the homotopy theorem in Section \ref{sec:homotopy}. 

\section{Definitions}\label{sec:def}
We fix an o-minimal expansion of a real closed field $\mathcal R=(R,0,1,+,\cdot,\ldots)$ throughout this paper.
The term `definable' means `definable in the o-minimal structure' in this paper.
A definable $C^r$ manifold means an affine definable $C^r$ manifold, where $r$ is a nonnegative integer. 
The notation $\cdfr(M)$ denotes the ring of definable $C^r$ functions on $M$.
We review the definition of the Whitehead group.
\begin{definition}
Let $A$ be a commutative ring.
The notation $\gl(n,A)$ denotes the general linear group of degree $n$ with entries in $A$.
We may assume that $\gl(n,A) \subset \gl(n+1,A)$ by identifying $g \in \gl(n,A)$ with $\left(\begin{array}{cc} g & 0\\ 0 & 1\end{array}\right) \in \gl(n+1,A)$.
Set $\gl(A) = \bigcup_{n=1}^{\infty} \gl(n,A)$.
The notation $\slin(A)$ is the normal subgroup of $\gl(n,A)$ consisting of matrices of determinant one.
The commutator subgroup $\comm(A)=[\gl(A),\gl(A)]$ of $\gl(A)$ is a normal subgroup of both $\gl(A)$ and $\slin(A)$.
The Whitehead group $K_1(A)$ of the ring $A$ is defined as $\gl(A)/\comm(A)$.
The notation $\sk(A)$ denotes the subgroup $\slin(A)/\comm(A)$ of $K_1(A)$.
The Whitney group $K_1(A)$ is isomorphic to $A^\times \times \sk(A)$, where $A^\times$ denotes the set of units in $A$, by \cite[Chapter III, Example 1.1.1]{Weibel}.
\end{definition}

\section{Equivalence of $\sk(\cdfr(M))$ with $\sk(\cdfz(M))$}\label{sec:main1}

Let $i$ and $j$ be distinct positive integers and $r \in A$.
The notation $e(i,j;r)$ denotes the elementary matrix which has one in every diagonal spot, has $r$ in the $(i,j)$-spot, and has zero elsewhere.
The commutator subgroup $\comm(A)$ of $\gl(A)$ is generated by the elements of the form $e(i,j;r)$ by \cite[Chapter III, Whitehead's Lemma 1.3.3]{Weibel}.
The assertions similar to the following three lemmas are found in \cite{Weibel}.
However, they are not summarized as lemmas.
\begin{lem}\label{lem:basic1}
Let $i$ and $j$ be distinct positive integers.
Let $c(i,j)$ be the matrix in $\gl(\mathbb Z)$ whose $(i,j)$-th entry is one, $(j,i)$-th entry is $-1$, 
the other entries in the $i$-th and $j$-th rows are zeros, and entries in the other rows coincide with those of the identity matrix.
The matrix $c(i,j)$ is an element of $\comm(\mathbb Z)$.
\end{lem}
\begin{proof}
We only prove the case in which $i=1$ and $j=2$.
The proof is similar in the other cases.
We have 
\begin{equation*}
e(1,2;1)e(2,1;-1)e(1,2;1)=
\left(\begin{array}{cc} 1 & 1\\  0 & 1 \end{array}\right)
\left(\begin{array}{cc} 1 & 0\\  -1 & 1 \end{array}\right)
\left(\begin{array}{cc} 1 & 1\\  0 & 1 \end{array}\right)
=
\left(\begin{array}{cc} 0 & 1\\  -1 & 0 \end{array}\right)\text{.}
\end{equation*}
\end{proof}

\begin{lem}\label{lem:basic2}
Let $A$ be a commutative ring.
Let $i$ and $j$ be distinct positive integers, and $r \in A^\times$.
Let $m(i,j;r)$ be the matrix in $\gl(A)$ whose $(i,i)$-th entry is $r$, $(j,j)$-th entry is $r^{-1}$, the other diagonal entries are ones and the non-diagonal entries are zeros.
The matrix $m(i,j;r)$ is an element of $\comm(A)$.
\end{lem}
\begin{proof}
We only prove the case in which $i=1$ and $j=2$.
The proof is similar in the other cases.
We have 
\begin{equation*}
e(1,2;r)e(2,1;-r^{-1})e(1,2;r)\left(\begin{array}{cc} 0 & -1\\  1 & 0 \end{array}\right) = \left(\begin{array}{cc} r & 0\\  0 & r^{-1} \end{array}\right) \text{.}
\end{equation*}
It means that $m(1,2;r) \in \comm(A)$ by Lemma \ref{lem:basic1}.
\end{proof}

\begin{lem}\label{lem:basic3}
Let $T$ be a topological space.
Consider a subring $B$ of the ring of $R$-valued functions on $T$.
Let $A=(a_{ij})$ be an $n \times n$-matrix with $|a_{ij}|< \frac{1}{n-1}$ on $T$ for all $1 \leq i,j \leq n$ whose entries are in $B$.
We further assume that $I+A \in \slin(n,B)$.
There exist a positive integer $N$, pairs of positive integers $\{(i_k,j_k)\}_{k=1}^N$, and  $\{r_k \in B\}_{k=1}^N$ satisfying the following equality:
\begin{equation*}
I+A = \prod_{k=1}^N e(i_k,j_k;r_k) \text{.}
\end{equation*}
Here, the notation $I$ denotes the identity matrix.
\end{lem}
\begin{proof}
We first diagonalize the matrix $I+A$ by multiplying matrices of the form $e(i,j;r)$ with $r \in B$.
Set $u=1+a_{11}$, then $|u|> 1 - \frac{1}{n-1}=\frac{n-2}{n-1}$.
Consider the inverse $v=u^{-1}$.
We have $|v| < \frac{n-1}{n-2}$.
For all $j \not=1$, subtract $va_{1j}$-times of the first column from the $j$-th column. 
This manipulation corresponds to the multiplication of $e(1,j;-va_{1j})$.
Let $I+A'$ be the result of the computation.
The first row of $I+A'$ is of the form $(u\ 0 \ \ldots \ 0)$.
Let $a'_{ij}$ be the $(i,j)$-th entry of $A'$.
We have
\begin{equation*}
|a_{ij}'|=|a_{ij}-va_{1j}a_{i1}| \leq |a_{ij}| + |va_{1j}a_{i1}| < \frac{1}{n-1}+\frac{n-1}{n-2} \left(\frac{1}{n-1}\right)^2=\frac{1}{n-2}
 \end{equation*}
 on $T$.
 Continuing this process, we can diagonalize $I+A$. 
 We can get the identity matrix by multiplying the diagonal matrix with matrices of the form $m(i,j;r)$.
They are the product of matrices of the form $e(i,j;r)$ by Lemma \ref{lem:basic2}.
\end{proof}

Lemma \ref{lem:basic3} claims that a matrix sufficiently close to the identity matrix is an element of $\comm(\cdfr(M))$.
It is very useful.

Let $M$ be a definable $C^r$ manifold, where $r$ is a nonnegative integer.
The Whitney group $K_1(\cdfr(M))$ is isomorphic to the group $(\cdfr(M))^\times \times \sk(\cdfr(M))$ by \cite[Chapter III, Example 1.1.1]{Weibel}.
The Whitney group $K_1(\cdfz(M))$ is isomorphic to the group $(\cdfz(M))^\times \times \sk(\cdfz(M))$ for the same reason.
In general, two groups $(\cdfr(M))^\times$ and $(\cdfz(M))^\times$ are not isomorphic.
However, the following theorem guarantees that the second components $\sk(\cdfr(M))$ and $\sk(\cdfz(M))$ are isomorphic.

\begin{thm}\label{thm:isom_whitehead}
Let $M$ be a definable $C^r$ manifold, where $r$ is a nonnegative integer.
The group $\sk(\cdfr(M))$ is isomorphic to $\sk(\cdfz(M))$.
\end{thm}
\begin{proof}
Since definable $C^r$ functions on $M$ are simultaneously definable continuous functions,  
there is a natural homomorphism $\Psi:\sk(\cdfr(M)) \rightarrow \sk(\cdfz(M))$.

We first show that the homomorphism $\Psi$ is injective.
We have only to show that any element $A \in \slin(\cdfr(M)) \cap \comm(\cdfz(M))$ is an element of $\comm(\cdfr(M))$.
There are a positive integer $K$, pairs of positive integers $\{(i_k,j_k)\}_{k=1}^K$ and definable continuous functions $\{c_k\}_{k=1}^K$ on $M$ with
\begin{equation*}
A = \prod_{k=1}^K e(i_k,j_k;c_k) \text{.}
\end{equation*}
Set $e_k=e(i_k,j_k;c_k)$ for all $1 \leq k \leq K$.
Take a positive integer $n$ so that $i_k \leq n$ and $j_k \leq n$ for all $1 \leq k \leq K$.
The notation $a_{ij}(x)$ denotes the $(i,j)$-th entry of $A^{-1}$.
Let $\delta(x)$ be a definable continuous function on $M$ satisfying the inequality 
\begin{equation*}
\max_{1 \leq i,j \leq n, \ 1 \leq k \leq K}(1,|a_{ij}(x)|,|c_k(x)|) < \delta(x)\text{.}
\end{equation*}
Take a positive definable continuous function $\varepsilon(x)$ on $M$ such that $\varepsilon(x)<\delta(x)$ and
\begin{equation*}
2^{K-1}n^KK\delta(x)^{K}\varepsilon(x)<\frac{1}{n-1} \text{.}
\end{equation*}
Using \cite{Escribano}, we can take definable $C^r$ approximations $c'_k$ of $c_k$ satisfying 
\begin{equation*}
|c'_k(x)-c_k(x)| < \varepsilon(x)
\end{equation*}
for all $1 \leq k \leq K$.
We have $|c'_k(x)|<2\delta(x)$.
Set $e'_k=e(i_k,j_k;c'_k(x))$.
The notation $\|C\|_{\text{max}}$ denotes the maximum of all the absolute values of entries in $C$ for any matrix $C$.
Note that $\|CD\|_{\text{max}} \leq n\|C\|_{\text{max}} \cdot \|D\|_{\text{max}}$ for all $n \times n$-matrices $C$ and $D$.
We have 
\begin{align*}
&\left\| \prod_{l=1}^K e_l - \prod_{l=1}^K e'_l\right\|_{\text{max}}
\leq \left\| (e_1-e'_1) \cdot \left(\prod_{l=2}^K e_l \right)\right\|_{\text{max}}
+  \left\| e'_1 \cdot (e_2-e'_2) \cdot \left(\prod_{l=3}^K e_l \right)\right\|_{\text{max}}\\
&\qquad + \cdots + \left\| \left(\prod_{l=1}^{K-1} e'_l \right) \cdot (e_K-e'_K)\right\|_{\text{max}}
\leq Kn^{K-1}(2\delta(x))^{K-1}  \varepsilon(x) \text{.}
\end{align*}
We get 
\begin{align*}
\left\| I - A^{-1}\prod_{l=1}^K e'_l\right\|_{\text{max}} &= \left\|A^{-1} \cdot\left(\prod_{l=1}^K e_l - \prod_{l=1}^K e'_l\right)\right\|_{\text{max}}\\
& \leq \left\|A^{-1}\right\|_{\text{max}} \cdot\left\|\left(\prod_{l=1}^K e_l - \prod_{l=1}^K e'_l\right)\right\|_{\text{max}}
 \leq 2^{K-1}n^KK\delta(x)^{K}\varepsilon(x)\\
 &<\frac{1}{n-1}\text{.}
\end{align*}
Note that $A^{-1}\prod_{l=1}^K e'_l$ is an element of $\slin(\cdfr(M))$ because $A$ is an element of $\slin(\cdfr(M))$.
We have 
\begin{equation*}
A^{-1}\prod_{l=1}^K e'_l \in \comm(\cdfr(M))
\end{equation*}
by Lemma \ref{lem:basic3}.
We have shown that $A \in \comm(\cdfr(M))$.
The homomorphism $\Psi$ is injective.

We next show that $\Psi$ is surjective.
Let $A \in \slin(\cdfz(M))$.
We assume that $A$ is an $n \times n$-matrix.
It suffices to show that $AE \in \slin(\cdfr(M))$ for some $E \in \comm(\cdfz(M))$.
Let $\delta(x)$ be a positive definable continuous function on $M$ with $\|A^{-1}\|_{\text{max}} < \delta(x)$.
Take a positive definable continuous function $\varepsilon(x)$ on $M$ with $n\delta(x)\varepsilon(x)<\frac{1}{n-1}$.
Since the special linear group $\slin(n,R)$ is a Nash submanifold of $ R^{n^2}$,
there exists $B \in \slin(\cdfr(M))$ with $\|B-A\|_{\text{max}}<\varepsilon(x)$ by \cite{Escribano}.
We have $\|A^{-1}B-I\|_{\text{max}} \leq n \|A^{-1}\|_{\text{max}} \cdot \|B-A\|_{\text{max}} < n\delta(x)\varepsilon(x)<\frac{1}{n-1}$.
Therefore, $E=A^{-1}B$ is an element of $\comm(\cdfz(M))$ by Lemma \ref{lem:basic3}.
We have $B=AE \in \slin(\cdfr(M))$.
\end{proof}

\section{Homotopy theorem}\label{sec:homotopy}
We demonstrate the homotopy theorem in this section.

\begin{lem}\label{lem:cover01}
Let $X$ be a definable subset of $ R^n$ and $\{V_j\}_{j=1}^p$ be a finite definable open covering of $X \times [0,1]$.
Then, there exist a finite definable open covering $\{U_i\}_{i=1}^q$ of $X$ and finite definable $C^r$ functions $0= \varphi_{i,0} < \cdots < \varphi_{i,k} < \cdots < \varphi_{i,r_i}=1$ on $U_i$ such that, for any $1 \leq i \leq q$ and $1 \leq k \leq r_i$, the definable set
\begin{equation*}
\{(x,t) \in U_i \times [0,1] \;|\; \varphi_{i,k-1}(x) \leq t \leq \varphi_{i,k}(x)\}
\end{equation*}
is contained in $V_j$ for some $1 \leq j \leq p$.
\end{lem}
\begin{proof}
The semialgebraic counterpart is found in \cite[Lemma 11.7.4]{BCR}.
The real spectrum is used in the proof of \cite[Lemma 11.7.4]{BCR} but it can be proved by using slicing without the real spectrum as pointed out in \cite[Remark 11.7.5]{BCR}.
We get the lemma using the $C^r$ cell decomposition theorem \cite[Chapter 7, Theorem 3.2 and Chapter 7, Exercise 3.3]{vdD} in place of semialgebraic slicing. 
The complete proof is found in \cite[Lemma 2.7]{Fujita}.
We omit the proof.
\end{proof}

It is possible to demonstrate the following two lemmas following a standard argument using the fact that any closed definable set is the zero set of a definable $C^r$ function \cite[Theorem C.11]{vdDM}.
For instance, they are found in \cite[Lemma 2.3, Lemma 2.5]{Fujita}.
We omit the proofs.
\begin{lem}[Fine definable open covering]\label{lem:covering}
Let $M$ be a definable $C^r$ manifold with $0 \leq r < \infty$. 
Let $\{U_i\}_{i=1}^q$ be a finite definable open covering of $M$.
For each $1 \leq i \leq q$, 
there exists a definable open subset $V_i$ of $M$ satisfying the following conditions:
\begin{itemize}
\item  the closure $\overline{V_i}$ in $M$ is contained in $U_i$ for each $1 \leq i \leq q$, and 
\item  the collection $\{V_i\}_{i=1}^q$ is again a finite definable open covering of $M$. 
\end{itemize}
\end{lem}

\begin{lem}\label{lem:sep}
Let $M$ be a definable $C^r$ manifold with $0 \leq r < \infty$. 
Let $X$ and $Y$ be closed definable subsets of $M$ with $X \cap Y = \emptyset$.
Then, there exists a definable $C^r$ function $f:M \rightarrow [0,1]$ with $f^{-1}(0)=X$ and $f^{-1}(1)=Y$.
\end{lem}

The following proposition on the Nash manifold $\slin(n, R)$  is used in the proof the homotopy theorem.
\begin{prop}\label{prop:nash2}
There exists a finite semialgebraic open covering $\{\mathcal U_i\}_{i=1}^q$ of the Nash submanifold $\slin(n, R)$ of $ R^{n^2}$ such that, for any $1 \leq i \leq q$, there exist 
\begin{itemize}
\item a positive integer $K_i$,
\item finite positive integers $i(i,k)$ and $j(i,k)$ with $i(i,k) \leq n$ and $j(i,k) \leq n$, and
\item finite Nash functions $a_{ik}(x)$ on $\mathcal U_i$ for all $1 \leq k \leq K_i$
\end{itemize}
such that, for any $A \in \mathcal U_i$, the equality $A=\prod_{k=1}^{K_i}e(i(i,k),j(i,k),a_{ik}(A))$ holds.
\end{prop}
\begin{proof}
Let $X:\slin(n, R) \rightarrow \slin(n, R)$ be the identity map.
Consider an arbitrary prime cone $\alpha \in \rspec(\nash(\slin(n, R)))$.
Here, the notation $\nash(\slin(n, R))$ denotes the ring of Nash functions on $\slin(n, R)$ and $\rspec(\nash(\slin(n, R)))$ is its real spectrum.
Their definitions are found in \cite{BCR}.
The notation $k(\alpha)$ denotes the real closure of the quotient field of $\nash(\slin(n, R))/\supp(\alpha)$ under the ordering induced from the prime cone $\alpha$.
We may naturally consider that $X$ is an element of $\slin(n, \nash(\slin(n, R)))$.
The notation $X(\alpha)$ denotes the matrix in $\slin(n, k(\alpha))$ whose entries are the residue classes of the entries of $X$.
Applying the Gaussian elimination process, we can take
\begin{itemize}
\item a positive integer $K'_\alpha$,
\item finite positive integers $i'_\alpha(k)$ and $j'_\alpha(k)$ with $i'_\alpha(k) \leq n$ and $j'_\alpha(k) \leq n$, and
\item finite elements $b_{\alpha,k} \in k(\alpha)$ for all $1 \leq k \leq K'_\alpha$
\end{itemize}
such that $X(\alpha) = \prod_{k=1}^{K'_\alpha} e(i'_\alpha(k),j'_\alpha(k);b_{\alpha,k})$.
By \cite[Proposition 8.8.3]{BCR}, there exist a semialgebraic open subset $V_\alpha$ of $\slin(n, R)$ and Nash functions $c'_{\alpha,k}:V_\alpha \rightarrow  R$ and $m_{\alpha,ij}:V_\alpha \rightarrow  R$ satisfying the following conditions:
\begin{itemize}
\item $\alpha \in \widetilde{V_\alpha}$, 
\item $c'_{\alpha,k}(\alpha)=b_\alpha$, 
\item $m_{\alpha,ij}(\alpha)=0$, and
\item $X = (I+M_\alpha) \cdot \prod_{k=1}^{K_\alpha} e(i'_\alpha(k),j'_\alpha(k);c'_{\alpha,k})$ on $V_\alpha$, where $M_\alpha$ is the matrix whose $(i,j)$-th entry is $m_{\alpha,ij}$.
\end{itemize}
Set $U_\alpha=\left\{x \in V_\alpha\;|\; |m_{\alpha,ij}(x)| < \frac{1}{n-1} \text{ for all } 1 \leq i,j \leq n\right\}$.
The set $\widetilde{U_\alpha}$ contains the point $\alpha$ because $m_{\alpha,ij}(\alpha)=0$.
Applying Lemma \ref{lem:basic3},  we can take
\begin{itemize}
\item a positive integer $K''_\alpha$,
\item finite positive integers $i''_\alpha(k)$ and $j''_\alpha(k)$ with $i''_\alpha(k) \leq n$ and $j''_\alpha(k) \leq n$, and
\item finite Nash functions $c''_{\alpha,k}$ on $U_\alpha$ for all $1 \leq k \leq K''_\alpha$
\end{itemize}
such that $I+M_\alpha = \prod_{k=1}^{K''_\alpha} e(i''_\alpha(k),j''_\alpha(k);c''_{\alpha,k})$ on $U_\alpha$.
We have constructed 
\begin{itemize}
\item a positive integer $K_\alpha$,
\item finite positive integers $i_\alpha(k)$ and $j_\alpha(k)$ with $i_\alpha(k) \leq n$ and $j_\alpha(k) \leq n$, and
\item finite Nash functions $c_{\alpha,k}$ on $U_\alpha$ for all $1 \leq k \leq K_\alpha$
\end{itemize}
such that the equality $X = \prod_{k=1}^{K_\alpha} e(i_\alpha(k),j_\alpha(k);c_{\alpha,k}(X))$ holds true on $U_\alpha$.

The family $\{U_{\alpha}\}_{\alpha \in \rspec(\nash(M))}$ is an open covering of $\rspec(\nash(M))$.
Since $\rspec(\nash(M))$ is compact by \cite[Corollary 7.1.13]{BCR}, a finite subfamily $\{U_{j}\}_{j=1}^r$ of $\{U_{\alpha}\}_{\alpha \in \rspec(\nash(M))}$ covers $\rspec(\nash(M))$.
The family $\{V_j\}_{j=1}^r$ covers $M$ by \cite[Proposition 7.2.2]{BCR}.
By the definition of $U_j$, we obtain the proposition.
\end{proof}

\begin{thm}\label{thm:whitehead_homotopy}
Let $M$ be a definable manifold.
Assume that two elements $A,B \in \slin(\cdfz(M))$ are definably homotopic, that is; there exists a definable continuous function $H:M \times [0,1] \rightarrow \slin(n, R)$ with $H( \cdot,0)=A(\cdot)$ and $H( \cdot,1)=B(\cdot)$.
Then, the images of $A$ and $B$ in $\sk(\cdfz(M))$ are identical.
\end{thm}
\begin{proof}
Let $\{\mathcal U_i\}_{i=1}^q$ be the finite semialgebraic open covering of $\slin(n, R)$ given in Proposition \ref{prop:nash2}.
Set $V_i=H^{-1}(\mathcal U_i)$ for all $1 \leq i \leq q$.
The collection $\{V_i\}_{i=1}^q$ is a finite definable open covering of $M \times [0,1]$.
Applying Lemma \ref{lem:cover01}, we get
\begin{itemize}
\item a finite definable open covering $\{U_i\}_{i=1}^p$ of $M$,
\item finite definable $C^r$ functions $0= \varphi_{i,0} < \cdots < \varphi_{i,j} < \cdots < \varphi_{i,r_i}=1$ on $U_i$,
\item a positive integer $K_{ij}$ for all $1 \leq i \leq p$ and $1 \leq j \leq r_i$,
\item finite positive integers $i(i,j,k)$ and $j(i,j,k)$ with $i(i,j,k) \leq n$ and $j(i,j,k) \leq n$, and
\item finite definable continuous functions $a_{ijk}$ on $A_{ij}$ for all $1 \leq i \leq p$, $1 \leq j \leq r_i$ and $1 \leq k \leq K_{ij}$
\end{itemize}
such that 
\begin{equation*}
H(x,t)=\prod_{k=1}^{K_{ij}}e(i(i,j,k),j(i,j,k),a_{ijk}(x,t)) \text{,}
\end{equation*}
for any $(x,t) \in A_{ij}$, where
\begin{equation*}
A_{ij}=\{(x,t) \in U_i \times [0,1]\;|\; \varphi_{i,j-1}(x) \leq t \leq \varphi_{i,j}(x)\}\text{.}
\end{equation*}

We fix an integer $i$ with $1 \leq i \leq p$.
We show the following claim:

\medskip
{\bf{Claim 1.}} There exist 
\begin{itemize}
\item a positive integer $K_i$,
\item finite positive integers $i(i,k)$ and $j(i,k)$ with $i(i,k) \leq n$ and $j(i,k) \leq n$, and
\item finite definable continuous functions $a_{ik}(x,t)$ on $U_i \times [0,1]$ for all $1 \leq k \leq K_i$
\end{itemize}
such that the equality $H(x,t)=\prod_{k=1}^{K_i}e(i(i,k),j(i,k),a_{ik}(x,t))$ holds for any $x \in U_i$ and $0 \leq t \leq 1$.
\medskip

We begin to prove Claim 1.
We first define a definable continuous function $b_{ijk}$ on $U_i \times [0,1]$ by
\begin{equation*}
b_{ijk}(x,t)=\left\{\begin{array}{ll}
a_{ijk}(x,\varphi_{i,j-1}(x)) & \text{ if } t < \varphi_{i,j-1}(x) \text{,}\\
a_{ijk}(x,t) & \text{ if }  \varphi_{i,j-1}(x) \leq t \leq \varphi_{i,j}(x)\text{,}\\
a_{ijk}(x,\varphi_{i,j}(x)) & \text{ if } t > \varphi_{i,j}(x) \text{.}
\end{array}\right.
\end{equation*}
Consider the definable continuous maps $H'_{i,j}: U_i \times [0,1] \rightarrow \slin(n, R)$ given by
\begin{equation*}
H'_{i,j}(x,t)= \prod_{k=1}^{K_{ij}} e(i(i,j,k),j(i,j,k),b_{ijk}(x,t))
\end{equation*}
for all $(x,t) \in U_i \times [0,1]$.
We have $H'_{i,j}(x,t)=H(x,t)$ for all $(x,t) \in A_{i,j}$.
Set 
\begin{equation*}
\widetilde{H_i}(x,t)=H'_{i,1}(x,t) \cdot \prod_{j=2}^{r_i} \left(H'_{i,j}(x, \varphi_{i,j-1}(x))^{-1} \cdot H'_{i,j}(x,t)\right) \text{.}
\end{equation*}
We demonstrate that $\widetilde{H_i}(x,t)=H(x,t)$ for all $(x,t) \in U_i \times [0,1]$.
Fix a point $(x,t)  \in U_i \times [0,1]$.
We have $\varphi_{i,j-1}(x) \leq t \leq \varphi_{i,j}(x)$ for some $1 \leq j \leq r_i$.
If $j'>j$, we have $H'_{i,j'}(x,t)=H'_{i,j'}(x, \varphi_{i,j'-1}(x))$.
If $j'<j$, we have $H'_{i,j'}(x,t)=H'_{i,j'}(x, \varphi_{i,j'}(x))=H(x,\varphi_{i,j'}(x))=H'_{i,j'+1}(x,\varphi_{i,j'}(x))$.
Therefore, we get 
\begin{align*}
\widetilde{H_i}(x,t)&=\prod_{j'=1}^{j-1} \left(H'_{i,j'}(x,t) \cdot H'_{i,j'+1}(x, \varphi_{i,j'}(x))^{-1}\right)\\
&\qquad  \cdot H'_{i,j}(x,t) \cdot \prod_{j'=j+1}^{r_i} \left(H'_{i,j'}(x, \varphi_{i,j'-1}(x))^{-1} \cdot H'_{i,j'}(x,t)\right) \\
&= I \cdot  H(x,t) \cdot I = H(x,t)\text{.}
\end{align*}
The map $\widetilde{H_i}$ is of the same form as the right hand of the equality in Claim 1.
We have finished the proof of Claim 1.
\medskip

For all $1 \leq i \leq p$, take definable open subsets $W_i$ and $W'_i$ of $M$ such that $\overline{W'_i} \subset W_i \subset \overline{W_i} \subset U_i$ 
and $\{W'_i\}_{i=1}^p$ is a definable open covering of $M$ by Lemma \ref{lem:covering}.
We next take a definable continuous function $\varphi_i:M \rightarrow [0,1]$ such that the restriction of $\varphi_i$ to $\overline{W'_i}$ is one and the restriction of $\varphi_i$ to the complement of $W_i$ is zero using Lemma \ref{lem:sep}.
Define definable continuous maps $s_i:M \times [0,1] \rightarrow [0,1]$ and $r_i: M \times [0,1] \rightarrow M \times [0,1]$ by $s_i(x,t)=\max(t,\varphi_i(x))$ and $r_i(x,t)=(x,s_i(x,t))$, respectively.
We show the following claim:

\medskip
{\bf{Claim 2.}} For any definable continuous function $\eta:M \rightarrow [0,1]$, the images of $H(x,\eta(x))$ and $H(r_i(x,\eta(x)))$ are identical in $\sk(\cdfz(M))$.
\medskip

We begin to prove Claim 2.
Take $K_i$, $i(i,k)$, $j(i,k)$ and $a_{ik}(x)$ as in Claim 1.
We may assume that $M$ is a definable $C^r$ submanifold of a Euclidean space $ R^n$ which is simultaneously closed in $ R^n$.
In fact, $M$ is a definable $C^r$ submaniold of a Euclidean space $ R^n$ because $M$ is affine.
Take a definable $C^r$ function $H$ on $ R^n$ with $H^{-1}(0)=\overline{M} \setminus M$ using \cite[Theorem C.11]{vdDM}.
The image of $M$ under the definable immersion $\iota:M \rightarrow  R^{n+1}$ given by $\iota(x)=\left(x,1/H(x)\right)$ is a definable closed subset of $ R^{n+1}$.

Consider the restriction of $a_{ik}(x,\eta(x))$ to $\overline{W_i}$.
There exists a definable continuous extension $\overline{a}_{ik}$ of this function to $M$ by the definable Tietze extension theorem \cite[Chapter 8, Corollary 3.10]{vdD} using the fact that $M$ is closed in $ R^n$.
Consider definable functions $c_{ik}$ on $M$ given by 
\begin{equation*}
c_{ik}(x)=\left\{\begin{array}{ll} a_{ik}(r_i(x,\eta(x))) & \text{ if } x \in \overline{W_i} \text{,}\\
\overline{a}_{ik}(x) & \text{ otherwise.}\end{array}\right.
\end{equation*}
They are continuous because $r_i(x,\eta(x))=(x,\eta(x))$ if $x \in M \setminus W_i$. 
Set 
\begin{align*}
G(x)&= \left(\prod_{k=1}^{K_i} e(i(i,K_i+1-k), j(i,K_i+1-k); -\overline{a}_{i,K_i+1-k}(x)) \right) \\
&\qquad \cdot \left(\prod_{k=1}^{K_i} e(i(i,k),j(i,k),c_{ik}(x)) \right) \text{.}
\end{align*}
It is obvious that $G(x) \in \sk(\cdfz(M))$.
We show that $G(x)=(H(x,\eta(x)))^{-1} \cdot H(r_i(x,\eta(x)))$.
When $x \in \overline{W_i}$, we have
\begin{align*}
G(x)&= \left(\prod_{k=1}^{K_i} e(i(i,K_i+1-k), j(i,K_i+1-k); -{a_{i,K_i+1-k}}(x,\eta(x))) \right) \\
&\qquad \cdot \left(\prod_{k=1}^{K_i} e(i(i,k),j(i,k),a_{i,k}(r_i(x,\eta(x)))) \right)\\
&= \left(\prod_{k=1}^{K_i} e(i(i,k),j(i,k),a_{i,k}(x,\eta(x))) \right)^{-1} \\
&\qquad \cdot \left(\prod_{k=1}^{K_i} e(i(i,k),j(i,k),a_{i,k}(r_i(x,\eta(x)))) \right)\\
&=(H(x,\eta(x)))^{-1} \cdot H(r_i(x,\eta(x))) \text{.}
\end{align*}
On the other hand, if $x \not\in \overline{W_i}$, we get
\begin{align*}
G(x)&= \left(\prod_{k=1}^{K_i} e(i(i,K_i+1-k), j(i,K_i+1-k); -\overline{a}_{i,K_i+1-k}(x)) \right) \\
&\qquad \cdot \left(\prod_{k=1}^{K_i} e(i(i,k),j(i,k),\overline{a}_{i,k}(x)) \right)\\
&= \left(\prod_{k=1}^{K_i} e(i(i,k),j(i,k),\overline{a}_{i,k}(x)) \right)^{-1} 
\cdot \left(\prod_{k=1}^{K_i} e(i(i,k),j(i,k),\overline{a}_{i,k}(x)) \right)\\
&=(H(x,\eta(x)))^{-1} \cdot H(x,\eta(x))= (H(x,\eta(x)))^{-1} \cdot H(r_i(x,\eta(x)))
\end{align*}
because we have $r_i(x,\eta(x))=(x,\eta(x))$ in this case.
We have shown that $G(x)=(H(x,\eta(x)))^{-1} \cdot H(r_i(x,\eta(x))) \in \comm(\cdfz(M))$.
We have proven Claim 2.

\medskip
Using Claim 2, we have $A=H(x,0) \equiv H(r_p \circ r_{p-1} \circ \cdots \circ r_1(x,0)) = H(x,1)=B$ modulo $\comm(\cdfz(M))$.
We have finished the proof of the theorem.
\end{proof}

\begin{cor}
The Whitehead group of $\cdfr( R^n)$ is isomorphic to the group $(\cdfr( R^n))^\times$.
\end{cor}
\begin{proof}
The corollary follows from Theorem \ref{thm:isom_whitehead} and Theorem \ref{thm:whitehead_homotopy} because $R^n$ is homotopy equivalent to a point.
\end{proof}

\end{document}